\documentclass[11pt]{amsart}
\usepackage{amssymb}
\usepackage[mathscr]{eucal}
\usepackage{hyperref}
\usepackage{url}
\usepackage[all]{xy}
\usepackage{tikz-cd}
\usepackage{array}
\usepackage[OT2,T1]{fontenc}
\DeclareSymbolFont{cyrletters}{OT2}{wncyr}{m}{n}
\DeclareMathSymbol{\Sha}{\mathalpha}{cyrletters}{"58}

\makeatletter
\let\@wraptoccontribs\wraptoccontribs
\makeatother

\newtheorem{thm}[equation]{Theorem}
\newtheorem{lem}[equation]{Lemma}
\newtheorem*{conjecture*}{Conjecture}

\newtheorem{cor}[equation]{Corollary}
\newtheorem{prop}[equation]{Proposition}
\newtheorem{conj}[equation]{Conjecture}

\theoremstyle{definition}
\newtheorem{rem}[equation]{Remark}
\newtheorem{defn}[equation]{Definition}

\newcommand{\rank}{\mathrm{rank}}
\numberwithin{equation}{section}

\def\Z{\mathbb{Z}}

\def\F{\mathbb{F}}

\def\A{\mathcal{A}}
\def\O{\mathcal{O}}
\def\P{\mathcal{P}}

\def\I{\mathcal{I}}

\def\cC{\mathcal{C}}

\def\p{\mathfrak{p}}

\def\q{\mathfrak{q}}

\def\Hom{\mathrm{Hom}}
\def\Gal{\mathrm{Gal}}

\def\Res{\mathrm{Res}}
\def\End{\mathrm{End}}

\def\Frob{\mathrm{Fr}}
\def\image{\mathrm{image}}
\def\Sel{\mathrm{Sel}}

\def\loc{\mathrm{loc}}
\def\cond{\mathrm{cond}}

\def\ur{\mathrm{ur}}

\def\GL{\mathrm{GL}}

\def\map#1{\;\xrightarrow{#1}\;}
\def\bmu{\boldsymbol{\mu}}
\def\too{\longrightarrow}

\def\dirsum#1{\underset{#1}{\textstyle\bigoplus}}

\def\Hu{H^1_\ur}

\title{$p$-twisted Selmer near-companion curves}

\author{Minseok Kim}
\address{Department of Mathematics, Yonsei University, Seoul 03722, Republic of Korea}
\email{\href{mailto:kms727@yonsei.ac.kr}{minseokkim.math@gmail.com}}
\begin{document}

\begin{abstract}
Let $E_1$ and $E_2$ be elliptic curves over a number field $K$. In \cite{scc}, Mazur and Rubin define the concept of $n$-Selmer near-companions and conjecture that if $E_1$ and $E_2$ are $n$-Selmer near-companions over $K$, then $E_1[n]$ is $G_K$-isomorphic to $E_2[n]$.
In \cite{2snc}, Yu proves the conjecture on $n$-Selmer near-companion curves in the case $n=2$. We introduce the notion of $p$-twisted Selmer near-companions ($p$-TSNC) over $K$ and prove that if $E_1$ and $E_2$ are $p$-TSNC over $K$, then $K(E_1[p])=K(E_2[p])$ under certain conditions.
\end{abstract}

\maketitle
\section{Introduction}

Let $p$ be a prime, and let $\F_p$ be the finite field with $p$-elements. For a $\F_p$-vector space $V$, we write $\dim_p(V)$ for the dimension of $V$ over $\F_p$.
For a field $k$, denote by $G(F/k)$ the Galois group of field extension $F/k$ and by $G_k$ the absolute Galois group.

Selmer groups of elliptic curves over number fields are fundamental objects in arithmetic geometry, offering deep insights into the Mordell–Weil group, the Birch and Swinnerton-Dyer conjecture, and Iwasawa theory. In \cite{scc}, Mazur and Rubin introduced the concept of Selmer near-companion curves, formalizing the idea that certain pairs of elliptic curves exhibit similar behavior in their Selmer groups when twisted by quadratic characters. They conjectured that if two elliptic curves $E_1$ and $E_2$ are 
$p$-Selmer near-companions, then $E_1[p]$ and $E_2[p]$ should be isomorphic as $G_K$-modules.

\begin{defn}
    Let $p$ be a prime number. Let $E_1$ and $E_2$ be elliptic curves over a number $K$. We say $E_1$ and $E_2$ are {\em $p$-Selmer near-companions} over $K$ if there exists a constant $C := C(E_1,{E_2},K)$ such that for every $\chi \in \Hom(G_K, \{\pm 1\}),$ $$\ | \dim_{p} \Sel_{p}(E_2^\chi/K) - \dim_{p} \Sel_{p}(E_1^\chi/K) | < C.$$ 
\end{defn}

\begin{conj}{\cite[Conjecture 7.15]{scc}}
    If $E_1$ and ${E_2}$ are $p$-Selmer near companions over $K$, then there exists a $G_K$-module isomorphism $E_1[p] \cong {E_2}[p]$.
\end{conj}
(The original conjecture is about $n$, but for simplicity, we only consider the case when $n=p$.)

    Mazur and Rubin proved the converse of this conjecture, requiring a stronger assumption specifically when $p=2$ or $p=3$, as follows.

\begin{thm}{\cite[Theorem 7.12]{scc}}
    Let $E_1$ and $E_2$ be elliptic curves over a number field $K$. Let $n_p=p^2$ if $p=2$ or $3$, and $n_p=p$ if $p>3$. If $E_1[n_p]\cong E_2[n_p]$ as $G_K$-modules, then $E_1$ and $E_2$ are $p$-Selmer near companions over $K$.
\end{thm}

Yu proved the conjecture for $p=2$ \cite[Theorem 2.1]{2snc}.

\begin{thm}
    Let $E_1$ and $E_2$ be elliptic curves over a number field $K$. If they are $2$-Selmer near companions over $K$, then there is a $G_K$-module isomorphism $E_1[2] \cong {E_2}[2]$.
\end{thm}

Yu proved the conjecture using the fact that \( K(E_1[2]) = K(E_2[2]) \) if and only if there exists a \( G_K \)-module isomorphism \( E_1[2] \cong E_2[2] \). However, we do not know whether a \( G_K \)-module isomorphism \( E_1[p] \cong E_2[p] \) necessarily exists when \( K(E_1[p]) = K(E_2[p]) \).

In this paper, we extend the framework of Mazur and Rubin by introducing the notion of $p$-twisted Selmer near-companion curves (abbreviated as $p$-TSNC). Instead of quadratic twists, we consider twists by all characters $\chi\in\Hom(G_K,\mu_p)$, and we compare the dimensions of the corresponding twisted $p$-Selmer groups.
\begin{defn}
We say $E_1$ and $E_2$ are {\em $p$-twisted Selmer near-companions ($p$-TSNC)} if there exists a constant $C := C(E_1,{E_2},K)$ such that for every $\chi \in \Hom(G_K, \mu_p),$ $$\ | \dim_{p} \Sel_{p}(E_2/K,\chi) - \dim_{p} \Sel_{p}(E_1/K,\chi) | < C.$$ 
\end{defn}

We are interested in elliptic curves that are $p$-TSNC over $K$. Our main theorems are as follows:

\begin{thm} \label{theorem}
Let $E_1$ and $E_2$ be elliptic curves over a number field $K$. Let $M_i:=K(E_i[p])$ for $i=1,2.$
Suppose that $M_1\ne M_2.$
Assume that one of the following conditions holds:
\begin{enumerate}
    \item $E_i[p]\subset E_i(K)$ for some $i=1,2$, or \label{H1}
    \item $E_i(K)[p]\cong \Z/p\Z$ for $i=1,2$, \label{H2}or
    \item $E_1(K)[p]\cong \Z/p\Z$ and $K(\mu_p)\subsetneq M_2$, or\label{H3}
    \item $[M_i:K(\mu_p)]\nmid p$ for $i=1,2$,\label{H4}
    \item $\mu_p\subset K.$ \label{H5}
\end{enumerate}
Then $E_1$ and $E_2$ are not $p$-TSNC over $K$. Equivalently, if $p$ satisfies one of the above conditions and if $E_1$ and $E_2$ are $p$-TSNC over $K$, then $K(E_1[p])=K(E_2[p]).$
\end{thm}

\begin{thm} \label{tsncthm}
    Let $E_1$ and $E_2$ be elliptic curves over a number field $K$. Assume that there is a $G_K$-module isomorphism $E_1[p]\cong E_2[p]$ for $p\ge3$. Then $E_1$ and $E_2$ are $p$-TSNC over $K$.
\end{thm}

Note that a celebrated paper of Serre~\cite{serre} shows that all but finitely 
many primes $p$ satisfy condition~\ref{H4}. More precisely, if $E/K$ has no complex 
multiplication, then
\[
  \rho_{E,p} : G_K \to \GL_2(\F_p)
\]
is surjective for all but finitely many primes $p$.
If $E/K$ has complex multiplication by $\mathcal O_M$, then for all but 
finitely many primes $p$, the image of
\[
  \rho : G_K \too \GL_2(\F_p)
\]
is contained in a Cartan subgroup $C_p \subset \GL_2(\F_p)$ or in its normalizer 
$N(C_p)$. Hence,
\[
  \Gal(K(E[p])/K) \cong C_p \quad \text{or} \quad N(C_p),
\]
where $C_p$ is isomorphic to 
$(\mathcal O_M/p\mathcal O_M)^\times$, 
$\F_p^\times \times \F_p^\times$, 
or $\F_{p^2}^\times$,
depending on the splitting behavior of $p$ in $M$.  
In all cases, the prime $p$ satisfies condition~\ref{H4}.

Suppose that $\mu_p \subset K$. Then any two elliptic curves $E_1$ and $E_2$
over $K$ satisfy at least one of the conditions \ref{H1}--\ref{H4}.
Suppose that $E_1$ and $E_2$ do not satisfy \ref{H4}, that is,
\[
[M_i : K(\mu_p)] = [M_i : K] \mid p
\quad \text{for some } i=1,2.
\]
We first consider $E_1$.
If $[M_1 : K]=1$, then $E_1[p]\subset E_1(K)$, and hence \ref{H1} holds.

If $[M_1 : K]=p$, then the image of the mod-$p$ Galois representation
\[
\rho_{E_1,p} : G_K \to \GL_2(\F_p)
\]
is a $p$-Sylow subgroup of $\GL_2(\F_p)$.
Since $\mu_p \subset K$, the determinant of $\rho_{E_1,p}$ is trivial, and hence
\[
G(M_1/K) \cong
\left\langle
\begin{pmatrix}
1 & 1\\
0 & 1
\end{pmatrix}
\right\rangle.
\]
In particular,
\[
E_1(K)[p] \cong \Z/p\Z.
\]
Now consider $E_2$.
If $[M_2 : K]=1$, then $E_2[p]\subset E_2(K)$, so \ref{H1} holds.
If $[M_2 : K]=p$, then the same argument shows that
\[
E_2(K)[p] \cong \Z/p\Z,
\]
and hence \ref{H2} holds. If $[M_2:K]\nmid p$, then \ref{H3} holds.
Therefore, when $\mu_p \subset K$, any two elliptic curves $E_1$ and $E_2$
satisfy at least one of the conditions \ref{H1}--\ref{H4}.

Let $E_1$ and $E_2$ be elliptic curves over a number field $K$ without complex multiplication. If $E_1$ is not isogenous to $E_2$ over $K$, then $K(E_1[p])\ne K(E_2[p])$ for infinitely many primes $p$, by Faltings’ theorem (see \cite[Theorem 8.7]{sscc}). In other words, if $E_1$ and $E_2$ are $p$-TSNC over $K$ for all but finitely many primes $p$, then they are isogenous over $K$.
If $E_1$ or $E_2$ has complex multiplication, the same conclusion holds under certain additional assumptions (see \cite{sscc}).

Moreover, we have the following observations:
\begin{rem}
    Let $\rho_i : G_K \too \GL_2(\F_p)$ be the mod $p$ Galois representation attached to elliptic curve $E_i$. Note that $E_1[p]\cong E_2[p]$ as $G_K$-modules if and only if there exists an element $A\in\GL_2(\F_p)$ such that $\rho_1(\sigma)=A^{-1}\rho_2(\sigma)A$ for all $\sigma\in G_K.$
\end{rem}

\begin{cor} \label{1rational}
    Let $K$ be a number field containing $\mu_p$. Suppose, for $i=1,2$, $E_i(K)[p]\cong \Z/p\Z.$ If $E_1$ and $E_2$ are $p$-TSNC over $K$, then $E_1[p]$ and $E_2[p]$ are $G_K$-isomorphic.
\end{cor}
\begin{proof}
    Assume that $\mu_p\subset K$ and $E_i(K)[p]\cong \Z/p\Z$ for $i=1,2$. If $E_1$ and $E_2$ are $p$-TSNC over $K$, by case \ref{H2} of Theorem \ref{theorem}, we have that $K(E_1[p])=K(E_2[p]).$
    Then we have $G(M_i/K)\cong \Z/p\Z$ for $i=1,2$.
    The Galois group $G(M_i/K)$ are $p$-Sylow subgroups of $\GL_2(\F_p)$. So, $\rho_i(G_K)$ is conjugate to $\left\langle\begin{pmatrix}
1 & 1 \\
0 & 1
\end{pmatrix}\right\rangle$. Note that \begin{align*}
    \begin{pmatrix}
1 & 1 \\
0 & 1
\end{pmatrix}=\begin{pmatrix}
1/m & 0 \\
0 & 1
\end{pmatrix}\begin{pmatrix}
1 & m \\
0 & 1
\end{pmatrix}\begin{pmatrix}
m & 0 \\
0 & 1
\end{pmatrix}.\end{align*} Thus, there exists $Q\in\GL_2(\F_p)$ such that $\rho_1(\sigma)=Q^{-1}\rho_2(\sigma)Q$ for all $\sigma\in G_K.$ Hence $\rho_1\sim\rho_2.$
\end{proof}

\begin{cor} \label{2ratioanal}
    Assume that $E_1[p]\subset E_1(K)$. If $E_1$ and $E_2$ are $p$-TSNC over $K$, then $E_1[p]$ and $E_2[p]$ are $G_K$-isomorphic.
\end{cor}
\begin{proof}
    If $E_1[p]\subset E_1(K)$ and if $E_1$ and $E_2$ are $p$-TSNC over $K$, then, by case \ref{H1} of Theorem \ref{theorem}, $K(E_2[p])=K$. Thus, $E_2(K)[p]=E_2[p]$. Hence $E_1[p]\cong E_2[p]$ with the trivial $G_K$-action. 
\end{proof}

Our strategy is as follows. Suppose that $K({E_1}[p])\ne K({E_2}[p])$. We view all Selmer groups $\Sel_p(E'/K,\chi)$ as subspaces of $H^1(K,E'[p])$ for $E'=E_1$ or $E_2$.
Then, we will divide all possible cases based on the field extension degrees, and in each case, by Chebotarev's density theorem, there exist infinitely many characters $\chi\in\Hom(G_K,\mu_p)$ such that, for every positive integer $d$,
\begin{align*}
    |\dim_p&(\Sel_p(E_1/K))-\dim_p(\Sel_p(E_2/K))| +2d \\ &\le |\dim_p(\Sel_p(E_1/K,\chi))-\dim_p(\Sel_p(E_2/K,\chi)) |.
\end{align*}

\section{Selmer groups}
In this section, we present the lemmas required for the proof of our main theorems. Although these lemmas are not original to this work, we have included them for the reader's convenience.
Fix a prime $p\ge 3$. Let $K$ be a number field and $v$ a place of $K$. Define $\mathcal{C}(K) := \Hom(G_K,\mu_p)$ and $\cC(K_v):=\Hom(G_{K_v},\mu_p)$. The subsequent definitions are from \cite[Definition 5.1 and 5.3]{kmr1}.

Let $\Sigma$ be a finite set of places of $K$ containing
all places where $E$ has bad reduction, all places dividing $p\infty$,
and sufficiently large such that
\begin{equation}\label{Sigma1}
\text{the primes in } \Sigma \text{ generate the ideal class group of } K,    
\end{equation}
\begin{equation}\label{Sigma2}
    \text{the natural map } \O_{K,\Sigma}^\times/(\O_{K,\Sigma}^\times)^p \to 
   \prod_{v\in\Sigma} K_v^\times/(K_v^\times)^p \text{ is injective,}
\end{equation}
where $$\O_{K,\Sigma}^\times:=\{x\in K^\times : x\in\O_{K_v}^\times\text{ for every }v\notin\Sigma\}.$$

\begin{rem}
 The set $\Sigma$ can always be enlarged to satisfy the above conditions, as shown in \cite[Lemma 6.1]{kmr1}.
\end{rem}
\begin{defn} \label{defnsel}
    Let $\chi \in \mathcal{C}(K)$ (resp. $\mathcal{C}(K_v)$) be nontrivial. Let $L$ denote the cyclic extension of $K$ (resp., $K_v$) corresponding to $\chi$. Define $\I_L:=\ker(\Z[G(L/K)]\to\Z)$. Then $\rank_\Z(\I_L)=p-1.$ We define the $p$-twist $E^\chi$ of $E$ to be the abelian variety $\I_L\otimes E$ of dimension $p-1$. Concretely, $$E^\chi = \ker(\Res^L_K(E) \rightarrow E)$$ where $\Res^L_K(E)$ denotes the Weil restriction of scalars of $E$ from $L$ to $K$. For details, see \cite{MRS}.

    Let $\O$ denote the ring of integers of the cyclotomic field of $p$-th roots of unity, and let $\p$ denote the unique prime of $\O$ lying above $p$. Note that $\O\subset\End_K(E^\chi)$ if $\chi$ is nontrivial(see \cite[Proposition 6.3.]{ds}).
    There exists a canonical $G_K$-isomorphism $E^\chi[\p] \cong E[p]$ (see \cite[Lemma 5.2.]{kmr1}).

    For a place $v$ of $K$, let $$\loc_v : H^1(K,E[p]) \too H^1(K_v,E[p])$$ denote the restriction map of group cohomology and if $c\in H^1(K,E[p]),$ denote $c_v:=\loc_v(c)$.

    For a place $v$ of $K$, let $\chi$ denote an element of $\cC(K_v)$. Define $$\gamma_v(\chi):=\image(E^\chi(K_v)/{\p E^\chi(K_v)} \rightarrow H^1(K_v,E^\chi[\p])\cong H^1(K_v,E[p])).$$ Note that this Kummer map depends on the choice of a generator $\p/\p^2$, but its image is independent of this choice.
    For a trivial character $1_v$, define $$\gamma_v(1_v):=\image(E(K_v)/pE(K_v)\hookrightarrow H^1(K_v,E[p]))$$ is the image of classical Kummer map.

    For a non-archimedean place $v$ with residue characteristic different from $p$, if $E$ has a good reduction at $v$, define $$\Hu(K_v,E[p]) := H^1(K_v^{\mathrm{ur}}/{K_v},E[p]),$$ where $K_v^{\ur}$ denotes the maximal unramified extension of $K_v$.

    For $\chi \in \cC(K)$, define $$\Sel_p(E/K,\chi):=\{c \in H^1(K,E[p]) : c_v \in \gamma_v(\chi_v) \text{ for all $v$} \},$$ where $\chi_v=\chi|_{G_{K_v}}.$

    Let $S$ be a set of primes of $K$. For $\psi \in \prod_{v\in S} \cC(K_v)$, define 
    \begin{align*}
      \Sel_p(E/K,\psi) := \{c \in H^1(K,E[p]) : c_v \in \gamma_v(\psi_v)& \text{ for $v\in S$}, \\&c_v \in \gamma_v(1_v) \text{ for $v\notin S$} \}.
    \end{align*}
    Note that if $\chi\in\cC(K)$ is trivial, $\Sel_p(E/K,\chi)$ is the classcial $p$-Selmer group of elliptic curve $E/K$. Thus, $$\Sel_p(E/K):=\{c \in H^1(K,E[p]) : c_v \in \gamma_v(1_v) \text{ for all $v$} \},$$
    Define $r_p(E):=\dim_p(\Sel_p(E/K))$ and $r_p(E,\chi):=\dim_p(\Sel_p(E/K,\chi)).$
\end{defn}

\begin{defn}
    For $1 \le i \le 2$, define \begin{align*}
    \P(E)& :=\{\q : \q\notin\Sigma\},\\
        \P_i(E)& := \{\q\in\P(E) : \mu_p \subset K_\q \text{  and } \dim_p \Hu(K_\q,E[p]) = i \},
    \end{align*} and define $\P_0(E) :=\{\q : \q \notin\Sigma\cup\P_1(E)\cup\P_2(E) \}.$
    Observe that $\P(E)=\P_0(E)\cup\P_1(E)\cup\P_2(E).$
\end{defn}

\begin{prop} \label{ramunr}
    Assume that $v\nmid p\infty$, $E$ has good reduction at $v$ and $\chi_v$ is ramified.
    Then $\gamma_v(\chi_v) \cap \Hu(K_v,E[p])=0$.
    \end{prop}
\begin{proof}
    See \cite[Proposition 7.8]{ds}.
\end{proof}

\begin{prop} \label{localcond}
    Assume that $v\nmid p\infty$, $E$ has good reduction at $v$ and $\chi_v$ is unramified.
    Then \begin{itemize}
        \item $\gamma_v(\chi_v)=\Hu(K_v,E[p])$,
        \item $\dim_p \gamma_v(\chi_v) = \dim_p E[p]^{\Frob_v=1}$, where $\Frob_v$ denotes the Frobenius generator,
        \item there exists an isomorphism $\Hu(K_v,E[p]) \cong E[p]/{(\Frob_v-1)E[p]}$ given by evaluating cocycles at $\Frob_v$.
    \end{itemize}
\end{prop}

\begin{proof}
    See \cite[Lemma 7.2 and 7.3]{ds}
\end{proof}

\begin{thm}\label{mod2}
    Let $\chi\in\cC(K)$. We have $$r_p(E,\chi)-r_p(E)\equiv \sum_{v}h_v(\chi_v) \text{   mod $2$},$$ where $\chi_v$ is the restriction of $\chi$ to $G_{K_v}$ and $$h_v(\chi_v):=\dim_p(\gamma_v(1_v)/({\gamma_v(\chi_v)\cap\gamma_v(1_v)})).$$
\end{thm}

\begin{proof}
    See \cite[Corollary 4.6]{alc}.    
\end{proof}
\begin{defn}
    Let $E/K$ be an elliptic curve over a number field $K$ and let $\q$ be a place of $K$. The relaxed twisted $p$-Selmer group $\Sel_p(E/K,\chi)^\q$ at $\q$ and the strict twisted $p$-Selmer group $\Sel_p(E/K,\chi)_\q$ at $\q$ are defined by the following exact sequences : 
\begin{equation*}
\raisebox{19pt}{
\xymatrix@C=12pt@R=7pt{
0 \ar[r] & \Sel_p(E/K,\chi)^\q \ar[r] & H^1(K,E[p]) \ar^-{\dirsum{v\ne\q}\loc_v}[rr] 
   && \dirsum{v\ne\q}\displaystyle\frac{H^1(K_v,E[p])}{\gamma_v(\chi_v)} \\
0 \ar[r] & \Sel_p(E/K,\chi)_{\q} \ar[r] & \Sel_p(E/K,\chi) \ar^-{\loc_\q}[rr] 
   && \gamma_\q(\chi_\q).
}}
\end{equation*}
In particular, if $\chi$ is trivial, $\Sel_p(E/K)^\q$ and $\Sel_p(E/K)_\q$ are defined as above.
\end{defn}

\begin{thm}\label{dual}
The images of the two right-hand maps in the following exact sequences are orthogonal complements of each other under the sum of the local Tate pairings.
\begin{equation}
\label{gdd}
\raisebox{19pt}{
\xymatrix@C=12pt@R=7pt{
0 \ar[r] & \Sel_p(E/K) \ar[r] & \Sel_p(E/K)^{\q} \ar^-{\loc_\q}[rr] 
   && \displaystyle\frac{H^1(K_\q,E[p])}{\gamma_\q(1_\q)} \\
0 \ar[r] & \Sel_p(E/K)_{\q} \ar[r] & \Sel_p(E/K) \ar^-{\loc_\q}[rr] 
   && \gamma_\q(1_\q).
}}
\end{equation}
In particular, $$\dim_p(\Sel_p(E/K)^{\q}) - \dim_p(\Sel_p(E/K)_{\q})=\dim_p(\gamma_\q(1_\q)) = \displaystyle\frac{1}{2}\dim_p(H^1(K_\q,E[p])).$$
\end{thm}

\begin{proof}
    See \cite[Theorem 2.3.4]{kolysys}
\end{proof}
\section{Increasing the Selmer rank}

In this section, let $E/K$ be an elliptic curve with $E(K)[p]\ne 0$ and let $n$ be any positive integer. We will prove that there exist infinitely many $\chi \in \cC(K)$ satisfying $r_p(E,\chi)=r_p(E)+2n$. In this section, we denote $M:=K(E[p])$.

The following lemma is frequently used to verify $\q \in \P_i(E)$.

\begin{lem} \label{plem}
   Let $\q$ be a prime of $K$ such that $\q \notin \Sigma$, and let $\Frob_{\q} \in G(K(E[p])/K)$ denote a Frobenius element for some choice of prime above $\q$. Then :
    \begin{enumerate}
        \item $\q \in \P_2(E)$ if and only if $\Frob_{\q} = 1$;
        \item $\q \in \P_1(E)$ if and only if $\Frob_{\q}$ has order exactly $p$;
        \item $\q \in \P_0(E)$ if and only if $\Frob_{\q}^{p} \ne 1$.
    \end{enumerate}
\end{lem}

\begin{proof}
    See \cite[Lemma 4.3]{kmr1}
\end{proof}

\begin{rem} \label{zerorem}
   By the Chebotarev density theorem, $\P_2(E)$ has positive density. If $\q \in \P_0(E)$ and $\psi_\q\in\mathcal{C}(K_\q)$, $$\dim_p\Hu(K_\q,E[p])=0=\dim_pH^1(K_\q,E[p]).$$ Thus, $\gamma_\q(1_\q)=\gamma_\q(\psi_\q)=0$ and $\Sel_p(E/K,\psi_\q)=\Sel_p(E/K)$.
\end{rem}

The following two lemmas are used to construct global characters from local characters.

\begin{lem} \cite[Lemma 6.6]{kmr1}
\label{elem}
Suppose $G$ and $H$ are abelian groups, and $J \subset G \times H$ is a subgroup. Let $\pi_G$ and $\pi_H$ denote the projection maps from $G \times H$ to $G$ and $H$, respectively.
Let $J_0 := \ker(J \map{\pi_G} G/G^p)$.
\begin{enumerate}
\item
The image of the natural map $\Hom((G \times H)/J,\bmu_p) \to \Hom(H,\bmu_p)$ is 
$\Hom(H/\pi_H(J_0),\bmu_p)$.
\item

If $J/J^p \to G/G^p$ is injective, 
then $\Hom((G \times H)/J,\bmu_p) \to \Hom(H,\bmu_p)$ 
is surjective.
\end{enumerate}
\end{lem}

\begin{lem} \label{cftlemma}
    Let $\Sigma$ be a (finite) set of places of $K$ such that $\mathrm{Pic}(\O_{K,\Sigma}) = 0$. Then the image of the restriction map \begin{align*}
        \cC(K)=\Hom(G_K,\mu_p)=\Hom((\prod_{v\in \Sigma}&K_v^\times \times \prod_{v \notin \Sigma} \O_v^\times)/\O_{K,\Sigma}^\times, \mu_p) \too \\ &\prod_{v \in \Sigma} \Hom(K_v^\times, \mu_p) \times \prod_{v \notin \Sigma} \Hom(\O_v^\times, \mu_p)
    \end{align*}
     is the set of all $((\phi_v)_{v\in\Sigma},(\psi_v)_{v\notin\Sigma})$ such that $\prod_{\p\in\Sigma}\phi_
     \p(b)\prod_{v\notin\Sigma}\psi_v(b) = 1$ for all $b\in\O_{K,\Sigma}^\times$ where $\phi_\p\in\Hom(K_\p,\mu_p)$ and $\psi_v\in\Hom(\O_{K_v},\mu_p).$
\end{lem}

\begin{proof}
Global class field theory, together with the assumption $\mathrm{Pic}(\O_{K,\Sigma})=0$, yields the equalities. The image follows as claimed.
\end{proof}

\begin{lem}
\label{4.6}
Suppose $p > 2$, $\q\in\P_2(E)$, and $\psi\in\cC(K_\q)$ is nontrivial. If $F$ is the cyclic extension of $K_\q$ corresponding to $\psi$, then
$$
\gamma_v(\psi) = \Hom(G(F/K_\q),E[p]) \subset \Hom(G_{K_\q},E[p]) = H^1(K_\q,E[p]).
$$
\end{lem}

\begin{proof}
    See \cite[Lemma 5.7]{kmr1}.
\end{proof}
\begin{rem} \label{canlem}
    In the proof of Lemma \ref{4.6}, the authors proved \begin{equation} \label{disp}
        E^\psi(K_\q)[p^\infty]=E^\psi[\p]\quad \text{and}\quad\dim_p(\Hom(G(F/K_\q),E[p]))=2.
    \end{equation} These results play a crucial role in increasing the Selmer rank.
\end{rem}

\begin{defn}\label{defnN}
For an elliptic curve $E/K$ and $\chi\in\cC(K)$, let $\bar{s}$ denote the image of $s \in \Sel_{p}(E/K,\chi)$ in the restriction map $$\Sel_{p}(E/K,\chi) \hookrightarrow H^1(K,E[p]) \too H^1(M,E[p])=\Hom(G_M,E[p]).$$
Let $L_{E,\chi}$ be the fixed field of $\cap_{s \in \Sel_{p}(E/K,\chi)} \ker(\bar{s})$. Let $N_E$ be the Galois closure of $L_{E,\chi} K(\sqrt[p]{\O_{K,\Sigma}^\times})$ over $K$. Observe that $N_E$ is a finite $p$-power extension of $M$.
\end{defn}

The following theorem appears as in \cite[Theorem 3.9]{msk}. For the reader’s convenience, we provide the full proof.
\begin{thm}\label{firstthm}
   Suppose that $E(K)[p]\ne 0$. Then, for every positive integer $n$, there exist infinitely many $\chi \in \cC(K)$ satisfying $r_p(E,\chi)=r_p(E)+2n$.
\end{thm}

\begin{proof}
    Fix $\chi'\in\cC(K)$. We claim that there exist infinitely many $\chi''\in\cC(K)$ such that $$r_p(E,\chi'')=r_p(E,\chi')+2.$$ The desired result then follows by induction.
    Let $\Sigma(\chi'):=\Sigma\cup\{\p\mid \cond(\chi')\}$.
    Let $\bar{s}$ denote the image of $s \in \Sel_{p}(E/K,\chi')$ in the restriction map $$\Sel_{p}(E/K,\chi') \hookrightarrow H^1(K,E[p]) \too \Hom(G_{M},E[p]).$$
    Let $N_E$ be the Galois closure of $L_{E,\chi'}K(\sqrt[p]{\O_{K,\Sigma(\chi')}^\times})$ over $K$ where $L_{E,\chi'}$ is the fixed field of $\cap_{s\in\Sel_p(E/K,\chi')}\ker(\bar{s})$. Choose a prime $\q\notin\Sigma(\chi')$ such that $\Frob_\q|_{N_E}=1$. Then $\q\in\P_2(E)$ and $\loc_\q(\Sel_p(E/K,\chi'))=0$ by Lemma \ref{plem}.
    Put $\psi \in \prod_{v \in \Sigma(\chi')} \Hom(K_v^\times, \mu_p) \times \prod_{v \notin \Sigma(\chi')} \Hom(\O_v^\times, \mu_p)$ so that 
    \begin{itemize}
        \item $\psi_v=1_v$ for $v\in \Sigma(\chi')$,
        \item $\psi_\q$ is not trivial, and
        \item $\psi_\p$ is trivial for $\p\notin\Sigma(\chi')\cup\{\q\}$.
    \end{itemize}
    Since $K(\sqrt[p]{\O_{K,\Sigma(\chi')}^\times}) \subset N_E$ and $\Frob_\q|_{N_E}=1$, we have $\psi_\q(\O_{K,\Sigma(\chi')}^\times)=1$. Hence, by Lemma \ref{cftlemma}, there exists $\chi \in \cC(K)$ such that \begin{itemize}
        \item $\chi_v=1_v$ for $v\in\Sigma(\chi')$,
        \item $\chi_\q$ is ramified,
        \item $\chi_\p$ is unramified for $\p\notin\Sigma(\chi')\cup\{\q\}$.
    \end{itemize}
     Observe that $$\Sel_p(E/K,\chi')=\Sel_p(E/K,\chi')_\q \subset \Sel_p(E/K,\chi'\chi) \subset \Sel_p(E/K,\chi')^\q.$$
    Consider the following two exact sequences : 
    \begin{equation}
\raisebox{19pt}{
\xymatrix@C=12pt@R=7pt{
0 \ar[r] & \Sel_p(E/K,\chi') \ar[r] & \Sel_p(E/K,\chi')^{\q} \ar^-{\loc_\q}[rr] 
   && \displaystyle\frac{H^1(K_\q,E[p])}{\gamma_\q(\chi'_\q)} \\
0 \ar[r] & \Sel_p(E/K,\chi')_{\q} \ar[r] & \Sel_p(E/K,\chi') \ar^-{\loc_\q}[rr] 
   && \gamma_\q(\chi'_\q).
}}
\end{equation}
Then the images of the two right-hand maps are orthogonal complements of each other under the sum of the local Tate pairings, by Poitou-Tate global duality. (See \cite[Theorem 2.3.4]{kolysys}.)
Since $\q\in\P_2(E), $$$\dim_p(\Sel_p(E/K,\chi')^{\q}) - \dim_p(\Sel_p(E/K,\chi')_{\q}) = \displaystyle\frac{1}{2}\dim_p(H^1(K_\q,E[p]))=2.$$
On the other hand, if $v\ne\q$, then $\gamma_v(\chi_v)=\gamma_v(\chi_v'\chi_v)$ by Proposition \ref{localcond}. Note that $\chi'_\q$ is unramified. Hence, by and Lemma \ref{4.6} and \cite[Theorem 1.4]{alc}, 
\begin{align*}
    r_p(E,\chi'\chi)-r_p(E,\chi') &\equiv \sum_{v\in S} \dim_p(\gamma_v(\chi_v')/\gamma_v(\chi_v'\chi_v)\cap\gamma_v(\chi'_v)) &\pmod{2} \\
    &\equiv \dim_p(\gamma_\q(\chi_\q')/\gamma_\q(\chi'_\q\chi_\q)\cap\gamma_\q(\chi_\q')) &\pmod{2}\\ 
    &\equiv \dim_p(\gamma_\q(\chi'_\q)) \equiv 0, &\pmod{2}
\end{align*} where $S$ is a set of primes of $K$ containing all primes above $p$, all primes dividing $\cond(\chi')\cond(\chi)$, and all primes where $E$ has bad reduction.
By our construction of $\q$, $\Sel_p(E,\chi')_\q=\Sel_p(E,\chi')$, and hence either $$ r_p(E,\chi'\chi)=r_p(E,\chi')\quad \text{or}\quad r_p(E,\chi'\chi)=r_p(E,\chi')+2.$$
Let $\chi'':=\chi'\chi$. 
Let $f$ be the composition
$$f : E^{\chi''}(K) \too E^{\chi''}(K)/\p E^{\chi''}(K) \too H^1(K,E^{\chi''}[\p]).$$
Since $E[p]\cong E^{\chi''}[\p]$ as $G_K$-modules, there exists $P\in E^{\chi''}(K)[\p]$ such that $P\ne 0$.
Observe that the following diagram commutes : \begin{equation*}
    \begin{tikzcd}[column sep=small]
        E^{\chi''}(K)[\p] \arrow[r,"f"]\arrow[d, hook] & \Sel_p(E/K,{\chi''}) 
        \arrow[r, phantom, "="] & \Sel_p(E/K,{\chi''}) \arrow[r, hook] & H^1(K,E[p]) \arrow[d, "\loc_\q"] \\
        E^{\chi''}[\p] \arrow[r,"\sim"] & E^{\chi''}(K_\q)/\p E^{\chi''}(K_\q) 
        \arrow[r, hook] & H^1(K_\q,E^{\chi''}[\p]) \arrow[r,"\sim"] & H^1(K_\q,E[p])
    \end{tikzcd}
\end{equation*}
    Note that $E^{\chi''}[\p]\cong E^{\chi''}(K_\q)/\p E^{\chi''}(K_\q)$ canonically by Remark \ref{canlem}. By diagram chasing, $\loc_\q(f(P))\ne 0$ and thus $\Sel_p(E/K,\chi') \subsetneq \Sel_p(E,\chi'')$. Hence $r_p(E,\chi'')=r_p(E,\chi')+2$.
\end{proof}
\section{proof of Theorem \ref{tsncthm}}

For the reader's convenience, we restate Theorem \ref{tsncthm} below.
\begin{thm}
    Let $E_1$ and $E_2$ be elliptic curves over a number field $K$. Assume that there is a $G_K$-module isomorphism $E_1[p]\cong E_2[p]$ for $p\ge 3$. Then $E_1$ and $E_2$ are $p$-TSNC over $K$.
\end{thm}

\begin{proof}
Let $S$ be a finite set of places of $K$ containing all places where $E_1$ or 
$E_2$ has bad reduction and all places above $p$ and~$\infty$. 
Fix $\chi \in \mathcal C(K)$. For $E \in \{E_1,E_2\}$, define
\[
  \Sel_p(E/K,\chi)^S
  := 
  \ker\!\left(
    H^1(K,E[p]) \longrightarrow 
    \bigoplus_{v \notin S}
      H^1(K_v,E[p])/\gamma_v(\chi_v)
  \right).
\]
Then we have an exact sequence
\[
  0 \longrightarrow \Sel_p(E/K,\chi)
  \longrightarrow \Sel_p(E/K,\chi)^S
  \longrightarrow 
  \bigoplus_{v\in S} H^1(K_v,E[p])/\gamma_v(\chi_v).
\]

We claim that $\gamma_v^1(\chi_v)=\gamma_v^2(\chi_v)$ for all $v\notin S$, where
\[
  \gamma_v^i(\chi_v)
  :=
  \operatorname{im}\!\left(
    E_i^\chi(K_v)/\p E_i^\chi(K_v)
    \longrightarrow 
    H^1(K_v,E_i^\chi[\p])
    \cong H^1(K_v,E_i[p])
  \right).
\]
    Using the given $G_K$-isomorphism $E_1[p]\cong E_2[p]$, we identify both with 
$E[p]$. In particular,
\[
  H^1(K_v,E_i[p]) = H^1(K_v,E[p])
  \quad\text{and}\quad
  H^1(K,E_i[p]) = H^1(K,E[p])
\]
for $i=1,2$. Hence if 
$\gamma_v^1(\chi_v)=\gamma_v^2(\chi_v)$ for all $v\notin S$, then
\[
  \Sel_p(E_1/K,\chi)^S = \Sel_p(E_2/K,\chi)^S.
\]

Set
\[
  C := \prod_{v\in S} \lvert H^1(K_v,E[p]) \rvert.
\]
Then $C$ is independent of $\chi$, and therefore $E_1$ and $E_2$ are 
$p$-TSNC over $K$.

    Now, if $\chi_v$ is unramified, then $\gamma_v^i(\chi_v)=\Hu(K_v,E[p])$ for $i=1,2$. Assume that $\chi_v$ is ramified. However, by \cite[Lemma 7.4.]{ds}, \begin{equation} \label{gammaeq}
        \gamma_v^i(\chi_v)\subset \ker(H^1(K_v,E[p]) \to H^1(F,E[p])),
    \end{equation} where $F$ is the cyclic extension of $K_v$ corresponding to $\chi_v$. Since $E/K_v$ has good reduction, the inertia group acts trivially on $E[p]$, we have $E[p]^{G_F}=E[p]^{G_{K_v}}$. Thus, \begin{align*}
        \ker(H^1(K_v,E[p])\too H^1(F,E[p]))\\=H^1(F/K_v,E[p]^{G_{K_v}})=&\Hom(G(F/K_v),E(K_v)[p]).
    \end{align*} By \cite[Lemma 7.2.]{ds}, since $$\dim_p(\gamma^i_v(\chi_v))=\dim_p(E(K_v)[p])=\dim_p(\Hom(G(F/K_v),E(K_v)[p])),$$ the inclusion \eqref{gammaeq} is equal. This implies that $\gamma_v^i(\chi_v)$ is independent of $i$. Hence, $\gamma_v^1(\chi_v)=\gamma_v^2(\chi_v)$ for all $v\notin S.$
\end{proof}


\section{The case \texorpdfstring{$E_1[p]\subset E_1(K)$}{E1[p] subset E1(K)}}
For the rest of this paper, we assume that $\Sigma$ is a finite set of places of $K$ 
containing all places where $E_1$ or $E_2$ has bad reduction, all places dividing $p\infty$, 
and is sufficiently large so that $\Sigma$ satisfies conditions \eqref{Sigma1} and \eqref{Sigma2}.  
For $\chi\in\mathcal{C}(K)$, set
$$\Sigma(\chi) := \Sigma \cup \{\mathfrak{p} \mid \cond(\chi)\}.$$

In this section, we assume that $E_1[p]\subset E_1(K)$ and we let $M := K(E_2[p])$.  
Recall that for any $\chi\in \mathcal{C}(K)$, the element $\bar{t}$ denotes the image of 
$t \in \Sel_p(E_2/K, \chi)$ under the restriction map
$$\Sel_p(E_2/K,\chi) \hookrightarrow H^1(K, E_2[p]) \longrightarrow \Hom(G_M, E_2[p]).$$

\begin{prop} \label{4.3}
    Suppose that $[M:K]\nmid p$. Fix $\chi_0\in\cC(K)$. Then there are infinitely many $\chi\in\cC(K)$ such that $r_p(E_1,\chi\chi_0)=r_p(E_1,\chi_0)+2$ and $r_p(E_2,\chi\chi_0)=r_p(E_2,\chi_0)$.
\end{prop}

\begin{proof}
    Fix $\chi_0.$
    Let $N$ be the Galois closure of $L_{E_1,\chi_0}K(\sqrt[p]{\O_{K,\Sigma(\chi_0)}^\times})$ over $K$. Then $N/K$ is a finite $p$-power extension. Since $[M:K]\nmid p$, $N$ and $M$ are linearly disjoint. Choose a prime $\q\notin\Sigma(\chi_0)$ such that $\Frob_\q|_N=1$ and $(\Frob_\q|_M)^p\ne 1$. Then $\loc_\q(\Sel_p(E_1/K,\chi_0))=0$ and $\q\in\P_2(E_1)\cap\P_0(E_2)$ by Lemma \ref{plem}. Put $\psi \in \prod_{v \in \Sigma(\chi_0)} \Hom(K_v^\times, \mu_p) \times \prod_{v \notin \Sigma(\chi_0)} \Hom(\O_v^\times, \mu_p)$ so that 
    \begin{itemize}
        \item $\psi_v=1_v$ for $v\in \Sigma(\chi_0)$,
        \item $\psi_\q$ is not trivial, and
        \item $\psi_\p$ is trivial for $\p\notin\Sigma(\chi_0)\cup\{\q\}$.
    \end{itemize}
    Since $K(\sqrt[p]{\O_{K,\Sigma(\chi_0)}^\times}) \subset N$ and $\Frob_\q|_{N}=1$, we have $\psi_\q(\O_{K,\Sigma(\chi_0)}^\times)=1$. Hence, by Lemma \ref{cftlemma}, there exists $\chi \in \cC(K)$ such that \begin{itemize}
        \item $\chi_v=1_v$ for $v\in\Sigma(\chi_0)$,
        \item $\chi_\q$ is ramified,
        \item $\chi_\p$ is unramified for $\p\notin\Sigma(\chi_0)\cup\{\q\}$.
    \end{itemize} 
    Since $\loc_\q(\Sel_p(E_1/K,\chi_0))=0$, by the proof of Theorem \ref{firstthm}, 
    $$r_p(E_1,\chi\chi_0)=r_p(E_1,\chi_0)+2.$$
    By noting that $\q\in\P_0(E_2)$ and Remark \ref{zerorem}, we also have that $$r_p(E_2,\chi\chi_0)=r_p(E_2,\chi_0). \qedhere$$
\end{proof}

Moreover, using apply Proposition \ref{4.3} and using induction, we have following corollary.

\begin{cor} \label{cor1}
    If $[M:K]\nmid p$, then there are infinitely many $\chi\in\cC(K)$ such that $r_p(E_1,\chi)=r_p(E_1)+2d$ and $r_p(E_2,\chi)=r_p(E_2)$.
\end{cor}

\begin{prop} \label{4.1}
    Suppose that $[M:K]=p$. Fix $\chi_0\in\cC(K)$. Then there exist infinitely many $\chi\in\cC(K)$ such that \begin{enumerate}
        \item $r_p(E_1,\chi\chi_0)=r_p(E_1,\chi_0)+2$,
        \item $r_p(E_2,\chi\chi_0)=r_p(E_2,\chi_0)+2$,
        \item \label{kerprop1}there exists $t\in\Sel_p(E_2/K,\chi\chi_0)$ so that the fixed field of $\ker(\bar{t})$ is not contained in $K(\sqrt[p]{K^\times}).$ 
    \end{enumerate}
\end{prop}
\begin{proof}
    Fix $\chi_0.$
    Let $N$ be the Galois closure of $L_{E_1,\chi_0}L_{E_2,\chi_0}K(\sqrt[p]{\O_{K,\Sigma(\chi_0)}^\times})$ over $K$. Then $N/K$ is a finite $p$-power extension. Choose a prime $\q\notin\Sigma(\chi_0)$ such that $\Frob_\q|_N=1$. Then $\loc_\q(\Sel_p(E_i/K,\chi_0))=0$ for $i=1,2$ and $\q\in\P_2(E_1)\cap\P_2(E_2)$ by Lemma \ref{plem}. Put $$\psi \in \prod_{v \in \Sigma(\chi_0)} \Hom(K_v^\times, \mu_p) \times \prod_{v \notin \Sigma(\chi_0)} \Hom(\O_v^\times, \mu_p)$$ so that 
    \begin{itemize}
        \item $\psi_v=1_v$ for $v\in \Sigma(\chi_0)$,
        \item $\psi_\q$ is not trivial, and
        \item $\psi_\p$ is trivial for $\p\notin\Sigma(\chi_0)\cup\{\q\}$.
    \end{itemize}
    Since $K(\sqrt[p]{\O_{K,\Sigma(\chi_0)}^\times}) \subset N$ and $\Frob_\q|_{N}=1$, we have $\psi_\q(\O_{K,\Sigma(\chi_0)}^\times)=1$. Hence, by Lemma \ref{cftlemma}, there exists $\chi \in \cC(K)$ such that \begin{itemize}
        \item $\chi_v=1_v$ for $v\in\Sigma(\chi_0)$,
        \item $\chi_\q$ is ramified,
        \item $\chi_\p$ is unramified for $\p\notin\Sigma(\chi_0)\cup\{\q\}$.
    \end{itemize} Then, by the proof of Theorem \ref{firstthm}, $r_p(E_i/K,\chi\chi_0)=r_p(E_i/K,\chi_0)+2$.
    In this setting, by Lemma \ref{4.6}, we have the following exact sequence:
    $$0 \too \Sel_p(E_2/K,\chi_0)\too\Sel_p(E_2/K,\chi\chi_0)\too \Hom(G(L_\q/K_\q),E_2[p])\to 0,$$ where $L_\q/K_\q$ is the cyclic extension corresponding to $\chi_\q.$ 
    Since $$r_p(E_2,\chi\chi_0)=r_p(E_2,\chi_0)+2 \quad\text{and} \quad \dim_p(\Hom(G(L_\q/K_\q),E_2[p]))=2,$$ the right-hand map is surjective, clearly. Let $t\in\Sel_p(E_2/K,\chi\chi_0)$ be such that $\image(\loc_\q(t))$ contains $Q\in E_2[p]$, which is not in $E_2(K)[p]$ and let $F$ denote the fixed field of $\ker(\bar{t}).$ 
    Since $\Frob_\q|_M=1$, $\loc_\q(\bar{t})=\loc_\q(t)$. Thus, $\image(\bar{t})$ contains $Q$. 
    On the other hand, since $t\in H^1(K,E_2[p]),$
    $\bar{t}$ is $G_K$-equivariant, and in particular the image of $\bar{t}$ is stable under $G_K$. Thus, $\bar{t}$ is surjective. 
    Indeed, let $\{P_1, P_2\}$ be a basis of $E_2[p]$ with $E_2(K)[p]\cong \Z/p\Z\cong\langle P_1\rangle.$ Let $Q=aP_1+bP_2$ with $b\ne 0.$ Then, for $\tau=\begin{bmatrix}
        1 & 1 \\
        0 & 1
    \end{bmatrix}$ in $G(M/K)$, $\tau(Q)=(a+b)P_1+bP_2$ is in the image of $\bar{t}.$ This implies that the image of $\bar{t}$ contains $bP_1$. Thus, the image of $\bar{t}$ contains $P_1$. Therefore, $\image(\bar{t})=E_2[p].$

    Now, we have an exact sequence $$0\too E_2[p]\cong G(F/M)\too G(F/K) \too G(M/K)\cong \Z/p\Z\too 0.$$ We claim that $G(F/K)$ is nonabelian. Let $g\in G(F/K)$ be an element such that $\bar{g}=\tau\in G(M/K)$, where $\tau$ acts like $\begin{bmatrix}
        1 & 1 \\
        0 & 1
    \end{bmatrix}$ on $E_2[p]$ by choosing an $\F_p$-basis of $E_2[p]$. Consider $c_g : G(F/M)\too G(F/M), \sigma \mapsto g\sigma g^{-1}$. Since $G(F/M)$ is normal in $G(F/K)$, the map is well-defined. Let $\sigma_2$ be the element in $G(F/M)$ such that $\bar{t}(\sigma_2)=P_2.$
    Then $$\bar{t}(g\sigma_2 g^{-1})=g\cdot\bar{t}(\sigma_2)\ne P_2=\bar{t}(\sigma_2).$$ Thus, $$g\sigma_2g^{-1}\ne \sigma_2.$$ This implies that $G(F/K)$ is not abelian. Therefore, the fixed field $\ker(\bar{t})$ is not contained in $K(\sqrt[p]{K^\times}).$
\end{proof}

\begin{lem} \label{nonzeroloc}
    Let $E/K$ be an elliptic curve and let $\q\in\P_2(E).$ Assume that $\loc_\q(\Sel_p(E/K))\ne 0.$ Then for any ramified character $\psi_\q\in\cC(K)$, $$r_p(E,\psi_\q)\le r_p(E).$$
\end{lem}

\begin{proof}
    Note that $\Sel_p(E/K)_\q\subset\Sel_p(E/K),\Sel_p(E/K,\psi_\q)\subset\Sel_p(E/K)^\q.$ Let $$C:=\Sel_p(E/K)/\Sel_p(E/K)_\q \quad \text{and} \quad D:=\Sel_p(E/K,\psi_\q)/\Sel_p(E/K)_\q.$$ By Theorem \ref{dual}, $\dim_p(C+D)\le 2.$ By Proposition \ref{ramunr}, $\dim_p(C\cap D)=0.$ Since $\loc_\q(\Sel_p(E/K))\ne 0$, $\dim_p(C)\ge1$ and thus, $\dim_p(D)\le 1.$ Hence $r_p(E,\psi_\q)\le r_p(E).$
\end{proof}

\begin{prop} \label{4.2}
    Suppose that $[M:K]=p$. Fix $\chi_0\in\cC(K)$. Then there are infinitely many $\chi\in\cC(K)$ such that $r_p(E_1,\chi\chi_0)=r_p(E_1,\chi_0)+4$ and $r_p(E_2,\chi\chi_0)\le r_p(E_2,\chi_0)+2$.
\end{prop}

\begin{proof}
    Fix $\chi_0.$
    Let $\chi_1\in\cC(K)$ be the character as in Proposition \ref{4.1} and let $t\in\Sel_p(E_2/K,\chi_1\chi_0)$ be the element such that the fixed field $F$ of $\ker(\bar{t})$ is not contained in $K(\sqrt[p]{K^\times})$. Let $L$ be the Galois closure of $F$ over $K$. Let $L_{E_1}'$ be the fixed field of $\cap_{s\in\Sel_p(E_1/K,\chi_1\chi_0)}\ker({s})$ and let $N$ be the Galois closure of $L_{E_1}'K(\sqrt[p]{\O_{K,\Sigma(\chi_1\chi_0)}^\times})$ over $K$. Since $NM\subset K(\sqrt[p]{K^\times})$ and $L\not\subset K(\sqrt[p]{K^\times})$, $NM$ and $L$ are linearly disjoint. Choose a prime $\q\notin\Sigma(\chi_1\chi_0)$ such that $\Frob_\q|_{NM}=1$ and $\Frob_\q|_F\ne 1$. Then $\q\in\P_2(E_1)\cap\P_2(E_2)$, $\loc_\q(\Sel_p(E_1/K,\chi_1\chi_0))=0$ and $\loc_\q(\Sel_p(E_2/K,\chi_1\chi_0))\ne 0$. 
    Put $$\psi \in \prod_{v \in \Sigma(\chi_1\chi_0)} \Hom(K_v^\times, \mu_p) \times \prod_{v \notin \Sigma(\chi_1\chi_0)} \Hom(\O_v^\times, \mu_p)$$ so that 
    \begin{itemize}
        \item $\psi_v=1_v$ for $v\in \Sigma(\chi_1\chi_0)$,
        \item $\psi_\q$ is not trivial, and
        \item $\psi_\p$ is trivial for $\p\notin\Sigma(\chi_1\chi_0)\cup\{\q\}$.
    \end{itemize}
    Since $K(\sqrt[p]{\O_{K,\Sigma(\chi_1\chi_0)}^\times}) \subset N$ and $\Frob_\q|_{N}=1$, we have $\psi_\q(\O_{K,\Sigma(\chi_1\chi_0)}^\times)=1$. Hence, by Lemma \ref{cftlemma}, there exists $\chi \in \cC(K)$ such that \begin{itemize}
        \item $\chi_v=1_v$ for $v\in\Sigma(\chi_1\chi_0)$,
        \item $\chi_\q$ is ramified,
        \item $\chi_\p$ is unramified for $\p\notin\Sigma(\chi_1\chi_0)\cup\{\q\}$.
    \end{itemize} By the proof of Theorem \ref{firstthm}, \begin{align*}
        r_p(E_1,\chi_2\chi_1\chi_0)=r_p(E_1,\chi_1\chi_0)+2=r_p(E_1,\chi_0)+4.
        \end{align*}
    Since $\loc_\q(\Sel_2(E/K,\chi_1\chi_0))\ne 0$, by Lemma \ref{nonzeroloc}, one can show that
    $$r_p(E_2,\chi_2\chi_1\chi_0) \le r_p(E_2,\chi_1\chi_0)=r_p(E_2,\chi_0)+2.\qedhere$$
\end{proof}

\begin{thm} \label{2ndthm}
    Suppose that $E_1[p]\subset E_1(K)$ and that $M\ne K$. Then, for any positive integer $d$, there are infinitely many $\chi\in\cC(K)$ such that $$r_p(E_1,\chi)-r_p(E_2,\chi)\ge r_p(E_1)-r_p(E_2)+2d.$$
\end{thm}

\begin{proof}
    Suppose that $M\ne K$. If $[M:K]\nmid p$, apply Corollary \ref{cor1}. If $[M:K]\mid p$, apply Proposition \ref{4.2} and use induction. 
\end{proof}
\section{The case $E_i(K)[p]\cong \Z/p\Z$ for $i=1,2$}

For the rest of paper, let $M_1:=K(E_1[p])$ and let $M_2:=K(E_2[p])$. Note that $[M_i:K(\mu_p)]\mid p$ for $i=1,2$ if $E_i(K)[p]\cong \Z/p\Z.$

\begin{defn}
    Let $\tilde{t}$ denote the image of $t\in\Sel_p(E_2/K,\chi)$ in the restriction map $$\Sel_p(E_2/K,\chi) \hookrightarrow H^1(K,E_2[p]) \too \Hom(G_{M_1M_2},E_2[p]).$$
\end{defn}

\begin{prop} \label{5.1}
    Suppose that $M_1\ne M_2$ and $[M_1:K]\le [M_2:K]$.
    Fix $\chi_0\in\cC(K)$. Then there exist infinitely many $\chi\in\cC(K)$ such that \begin{enumerate}
        \item $r_p(E_1,\chi\chi_0)=r_p(E_1,\chi_0)+2$,
        \item $r_p(E_2,\chi\chi_0)=r_p(E_2,\chi_0)+2$,
        \item there exists $t\in\Sel_p(E_2/K,\chi\chi_0)$ so that the fixed field of $\ker(\tilde{t})$ is not contained in $M_1(\sqrt[p]{M_1^\times}).$ \label{kerprop2}
    \end{enumerate}
\end{prop}

\begin{proof}
    Fix $\chi_0.$
    This follows by a similar argument as in the proof of Proposition \ref{4.1}. Let $N$ be the Galois closure of $L_{E_1,\chi_0}L_{E_2,\chi_0}K(\sqrt[p]{\O_{K,\Sigma(\chi_0)}^\times})$ over $K$. Choose a prime $\q\notin\Sigma(\chi_0)$ such that $\Frob_\q|_N=1$. Then $\q\in\P_2(E_1)\cap\P_2(E_2)$ and $\loc_\q(\Sel_p(E_i/K,\chi_0))=0$ for $i=1,2$. Then, as in the proof of Proposition \ref{4.1}, there exists a global character $\chi\in\cC(K)$ such that $\chi$ is ramified at $\q$ and that $r_p(E_i/K,\chi\chi_0)=r_p(E_i/K,\chi_0)+2$ for $i=1,2$. 
    
    Recall that we have the following exact sequence: $$0 \too \Sel_p(E_2/K,\chi_0)\too \Sel_p(E_2/K,\chi\chi_0) \too \Hom(G(L_\q/K_\q),E_2[p]) \too 0.$$ Let $t\in \Sel_p(E_2/K,\chi\chi_0)$ be an element such that $\image(\loc_\q(t))$ contains $P_2\in E_2[p]$ which is not in $E_2(K)[p]$. By a similar argument as in the proof of Proposition \ref{4.1}, since $\tilde{t}$ is $G_{M_1}$-equivariant (under the action of $G_K$), $\tilde{t}$ is a surjective map onto $E_2[p]$. Note that $G(M_1M_2/M_1)\cong G(M_2/M_1\cap M_2).$ Now, let $F$ be the fixed field of $\ker(\tilde{t})$. Then we have the following exact sequence $$0 \too E_2[p]\cong G(F/M_1M_2) \too G(F/M_1) \too G(M_1M_2/M_1) \too 0.$$ 
    Using a similar argument as in the proof of Proposition \ref{4.1}, one can show that $F/M_1$ is not abelian.
    Thus, $F\not\subset M_1(\sqrt[p]{M_1^\times}).$
\end{proof}

\begin{prop} \label{5.2}
    Suppose that $M_1\ne M_2$ and that $[M_1:K]\le[M_2:K].$ Fix $\chi_0\in\cC(K)$. Then there are infinitely many $\chi\in\cC(K)$ such that $$r_p(E_1,\chi\chi_0)=r_p(E_1,\chi_0)+4 \quad\text{and} \quad r_p(E_2,\chi\chi_0)\le r_p(E_2,\chi_0)+2.$$
\end{prop}

\begin{proof}
    Fix $\chi_0.$
    Let $\chi_1\in\cC(K)$ be the character as in Proposition \ref{5.1} and let $t\in\Sel_p(E_2/K,\chi_1\chi_0)$ be the element such that the fixed field $F$ of $\ker(\tilde{t})$ is not contained in $M_1(\sqrt[p]{M_1^\times})$. Let $L$ be the Galois closure of $F$ over $K$. Let $N$ be the Galois closure of $L_{E_1,\chi_1\chi_0}K(\sqrt[p]{\O_{K,\Sigma(\chi_1\chi_0)}^\times})$ over $K$. Since $NM_2\subset M_1(\sqrt[p]{M_1^\times})$ and $L\not\subset M_1(\sqrt[p]{M_1^\times})$, $NM_2$ and $L$ are linearly disjoint. Choose a prime $\q\notin\Sigma(\chi_1\chi_0)$ such that $\Frob_\q|_{NM_2}=1$ and $\Frob_\q|_F\ne 1$. Then $\q\in\P_2(E_1)\cap\P_2(E_2)$, $\loc_\q(\Sel_p(E_1/K,\chi_1\chi_0))=0$ and $\loc_\q(\Sel_p(E_2/K,\chi_1\chi_0))\ne0$. As in the proof of Proposition \ref{4.2}, let $\chi_2\in\cC(K)$ be the character such that \begin{itemize}
        \item $\chi_v=1_v$ for $v\in\Sigma(\chi_1\chi_0)$,
        \item $\chi_\q$ is ramified,
        \item $\chi_\p$ is unramified for $\p\notin\Sigma(\chi_1\chi_0)\cup\{\q\}$.
    \end{itemize} Then \begin{align*}
   &r_p(E_1,\chi_2\chi_1\chi_0)=r_p(E_1,\chi_0)+4\quad \text{and}\\&r_p(E_2,\chi_2\chi_1\chi_0) \le r_p(E_2,\chi_1\chi_0)=r_p(E_2,\chi_0)+2. \qedhere
    \end{align*} 
\end{proof}

\begin{thm} \label{3rdthm}
    Suppose that $M_1\ne M_2$ and $[M_1:K]\le [M_2:K]$. Then, for any positive integer $d$, there are infinitely many $\chi\in\cC(K)$ such that $$r_p(E_1,\chi)-r_p(E_2,\chi)\ge r_p(E_1)-r_p(E_2)+2d.$$
\end{thm}

\begin{proof}
    Suppose that $M_1 \ne M_2$. Apply Proposition \ref{5.2} and use induction.
\end{proof}
\section{The case $E_1(K)[p]\cong\Z/p\Z$ and $K(\mu_p)\subsetneq M_2.$}

Recall that if $E_1(K)[p]\cong\Z/p\Z$, then $[M_1:K(\mu_p)]\mid p$.

\begin{prop} \label{5ththm}
    Assume that $[M_2:K(\mu_p)]\nmid p$. Then, for any positive integer $d$, there exist infinitely many $\chi\in\cC(K)$ such that $$r_p(E_1,\chi)=r_p(E_1)+2d \quad\text{and} \quad r_p(E_2,\chi)=r_p(E_2).$$
\end{prop}

\begin{proof}
    We claim that for $\chi'\in\cC(K)$, there exist infinitely many $\chi\in\cC(K)$ such that \begin{align*}
        &r_p(E_1,\chi\chi')=r_p(E_1,\chi')+2\quad \text{and} \\
        &r_p(E_2,\chi\chi')=r_p(E_2,\chi').
    \end{align*}
    Then the result follows by induction.
    Fix $\chi'\in\cC(K)$. Let $\Sigma(\chi'):=\Sigma\cup\{\p\mid \cond(\chi')\}$.
    Let $\bar{s}$ denote the image of $s \in \Sel_{p}(E_1/K,\chi')$ in the restriction map $$\Sel_{p}(E_1/K,\chi') \hookrightarrow H^1(K,E_1[p]) \too \Hom(G_{M_1},E_1[p]).$$
    Let $N_{E_1}$ be the Galois closure of $L_{E_1,\chi'}K(\sqrt[p]{\O_{K,\Sigma(\chi')}^\times})$ over $K$. Observe that $N_{E_1}$ and $M_2$ are linearly disjoint.
    Choose a prime $\q\notin\Sigma(\chi')$ such that $\Frob_\q|_{N_{E_1}}=1$ and that $(\Frob_\q|_{M_2})^p\ne1$. Then $\q\in\P_2(E_1)\cap\P_0(E_2)$. As in the proof of Proposition \ref{4.3}, there exists $\chi \in \cC(K)$ such that \begin{itemize}
        \item $\chi_v=1_v$ for $v\in\Sigma(\chi')$,
        \item $\chi_\q$ is ramified,
        \item $\chi_\p$ is unramified for $\p\notin\Sigma(\chi')\cup\{\q\}$.
    \end{itemize} Then $r_p(E_1,\chi\chi')=r_p(E_1,\chi')+2$ and $r_p(E_2,\chi\chi')=r_p(E_2,\chi')$.
\end{proof}

\begin{prop}\label{pprop}
    Suppose that $M_1\ne M_2.$ Assume $[M_2:K(\mu_p)]= p.$ 
    Fix $\chi_0\in\cC(K)$. Then there exist infinitely many $\chi\in\cC(K)$ such that $$r_p(E_1,\chi\chi_0)=r_p(E_1,\chi_0)+4 \quad\text{and} \quad r_p(E_2,\chi\chi_0)\le r_p(E_2,\chi_0)+2.$$
\end{prop}

\begin{proof}
    Fix $\chi_0$. Let $N$ be the Galois closure of $L_{E_1,\chi_0}L_{E_2,\chi_0}K(\sqrt[p]{\O_{K,\Sigma(\chi_0)}^\times})$ over $K$. Choose a prime $\q\notin\Sigma(\chi_0)$ so that $\Frob_\q|_N=1.$ Then $\q\in\P_2(E_1)\cap\P_2(E_2)$ and $\loc_\q(\Sel_p(E_i/K,\chi_0))=0$ for $i=1,2$. Let ${\chi_1}$ be the global character such that \begin{itemize}
        \item ${\chi_1}_v=1_v$ for $v\in\Sigma(\chi_0)$,
        \item ${\chi_1}_\q$ is ramified,
        \item ${\chi_1}_\p$ is unramified for $\p\notin\Sigma(\chi_0)\cup\{\q\}$.
        \end{itemize} Since $E_1(K)[p]\ne 0$, we have that $r_p(E_1,{\chi_1}\chi_0)=r_p(E_1,\chi_0)+2$. By Theorem \ref{mod2} and Theorem \ref{dual}, we also have that $$r_p(E_2,{\chi_1}\chi_0)=r_p(E_2,\chi_0) \quad \text{or} \quad r_p(E_2,\chi_0)+2.$$
        If $r_p(E_2,{\chi_1}\chi_0)=r_p(E_2,\chi_0)$, repeating this process, we have that there exists $\chi\in\cC(K)$ such that 
    \begin{align*}
        &r_p(E_1,\chi\chi_1\chi_0)=r_p(E_1,\chi_0)+4 \quad \text{and}\\
        &r_p(E_2,\chi\chi_1\chi_0)\le r_p(E_2,\chi_0)+2.    
    \end{align*}
    Now, suppose that $r_p(E_2,\chi_1\chi_0)=r_p(E_2,\chi_0)+2.$
    Then, by Lemma \ref{4.6}, we have the following exact sequence:
    \begin{align*}
        0 \to \Sel_p(E_2/K,\chi_0) \to \Sel_p(E_2/K,\chi_1\chi_0) \to \Hom(G(L_\q/K_\q),E_2[p])\to 0.
    \end{align*}
    Let $\{P_1,P_2\}$ be a basis of $E_2[p]$ such that $E_2(K(\mu_p))=\langle P_1\rangle.$
    Let $t\in\Sel_p(E_2/K,\chi_1\chi_0)$ be such that $\image(\loc_\q(t))$ contains $P_2.$ Let $\tilde{t}$ be the image of $t$ in $H^1(M_1M_2,E_2[p])=\Hom(G_{M_1M_2},E_2[p]).$ Since $G(M_1M_2/M_1)\cong G(M_2/M_1\cap M_2)$ acts like $\begin{bmatrix}
        1 & 1 \\
        0 & 1
    \end{bmatrix}$ on $E_2[p]$ (with respect to basis $\{P_1, P_2\}$), and since $\tilde{t}$ is $G_K$-equivariant, $\tilde{t}$ is surjective.
    Let $F$ be the fixed field of $\ker(\tilde{t}).$ Then we have the following exact sequence $$0 \too E_2[p]\cong G(F/M_1M_2) \too G(F/M_1) \too G(M_1M_2/M_1) \too 0.$$ Note that $G(M_1M_2/M_1)\cong G(M_2/M_1\cap M_2).$
    Using a similar argument as in the proof of Proposition \ref{4.1}, one can show that $G(F/M_1)$ is not abelian. Thus, $F\not\subset M_1(\sqrt[p]{M_1^\times}).$ 
    Now, let $\chi':=\chi_1\chi_0$ and let $N$ be the Galois closure of $L_{E_1,\chi'} K( \sqrt[p]{\O_{K,\Sigma(\chi')}^\times}).$ Then $NM_2\subset M_1(\sqrt[p]{M_1^\times})$ and $F\not\subset M_1(\sqrt[p]{M_1^\times}).$ Choose $\q_1\notin\Sigma(\chi')$ such that $\Frob_{\q_1}|_{NM_2}=1$ and $\Frob_{\q_1}|_F\ne 1.$ Then $\loc_{\q_1}(\Sel_p(E_1/K,\chi'))=0$ and $\loc_{\q_1}(\Sel_p(E_2/K,\chi'))\ne 0.$ Let $\chi$ be the global character such that 
    \begin{itemize}
        \item ${\chi}_v=1_v$ for $v\in\Sigma(\chi')$,
        \item ${\chi}_{\q_1}$ is ramified,
        \item ${\chi}_\p$ is unramified for $\p\notin\Sigma(\chi')\cup\{\q\}$.
        \end{itemize}
    Then $r_p(E_1,\chi\chi')=r_p(E_1,\chi')+2$ and $r_p(E_2,\chi\chi')\le r_p(E_2,\chi').$
    Therefore, there exist infinitely many characters $\chi_2\in \cC(K)$ such that 
    \begin{align*}
        &r_p(E_1,\chi_2\chi_0)=r_p(E_1,\chi_0)+4 \quad \text{and}
        \\
        &r_p(E_2,\chi_2\chi_0)\le r_p(E_2,\chi_0)+2. \qedhere
    \end{align*}
\end{proof}

\begin{thm}
    Assume that $E_1(K)[p]\cong \Z/p\Z$ and $K(\mu_p)\subsetneq M_2$. Suppose that $M_1\ne M_2.$  Then, for any positive integer $d$, there are infinitely many $\chi\in\cC(K)$ such that $$r_p(E_1,\chi)-r_p(E_2,\chi)\ge r_p(E_1)-r_p(E_2)+2d.$$
\end{thm}
\begin{proof}
    If $[M_2:K(\mu_p)]\nmid p$, then apply Proposition \ref{5ththm} and use induction. If $[M_2:K(\mu_p)]=p$, apply Proposition \ref{pprop} and use induction.
\end{proof}

\section{The case $[M_1:K(\mu_p)], [M_2:K(\mu_p)]\nmid p$}

Recall that $M_i=K(E_i[p])$ for $i=1,2$. In this section, we assume that $[M_1:K(\mu_p)], [M_2:K(\mu_p)]\nmid p$.
\begin{defn}
    Define \begin{align*}
        \A_1:=& \ker(K^\times/{(K^\times)^p} \rightarrow M_1^\times/{(M_1^\times)^p}),\\
        \A_2:=& \ker(K^\times/{(K^\times)^p} \rightarrow M_2^\times/{(M_2^\times)^p}),\\
        \A_3:=&\ker(\O_{K,\Sigma}^\times/{(\O_{K,\Sigma}^\times)^p} \rightarrow \prod_{\q\in\P_0} \O_\q^\times/{(\O_\q^\times)^p}),
    \end{align*} where $\P_0:=\P_0(E_1)\cap\P_0(E_2)$.
\end{defn}

\begin{rem}
    Since there is a natural injection $\O_{K,\Sigma}^\times/{(\O_{K,\Sigma}^\times)^p} \rightarrow  K^\times/{(K^\times)^p}$, we identify $\O_{K,\Sigma}^\times/{(\O_{K,\Sigma}^\times)^p}$ with its image in $K^\times/{(K^\times)^p}$.
    By \cite[Lemma 6.2]{kmr1}, for $i=1,2$, $\A_i$ is generated by an element $\Delta_i\in\O_{K,\Sigma}^\times$.
\end{rem}

\begin{lem} \label{Alem}
    Suppose that $M_1 \ne M_2$. Then $\A_3\subset \A_1\A_2$.
\end{lem}

\begin{proof}
    Let $x\in\A_3 \setminus \A_1\A_2$. Then, for $i=1,2$, $K(\mu_p,\sqrt[p]{x}),M_1$ and $M_2$ are linearly disjoint. Since $[M_i:K(\mu_p)]\nmid p$ for $i=1,2$, choose a prime $\q\notin\Sigma$ such that $\Frob_\q(\sqrt[p]{x})=\zeta\sqrt[p]{x}$ and that $(\Frob_\q|_{M_i})^p\ne1$ for $i=1,2$. Then $\q\in\P_0(E_1)\cap\P_0(E_2)$.
    However, since $x\in\A_3$, $\Frob_\q(\sqrt[p]{x})=\sqrt[p]{x}$. This is a contradiction.
\end{proof}

\begin{prop} \label{7prop1}
    Assume $[M_1:K(\mu_p)]\le [M_2:K(\mu_p)]$. Suppose $[M_2:M_1\cap M_2]\ne p$.
    Then, for any positive integer $d$, there exist infinitely many $\chi\in\cC(K)$ such that $$r_p(E_1,\chi)=r_p(E_1)+2d \quad\text{and} \quad r_p(E_2,\chi)=r_p(E_2).$$
\end{prop}
\begin{proof}
    We claim that for $\chi'\in\cC(K)$, there exist infinitely many $\chi\in\cC(K)$ such that \begin{align*}
        &r_p(E_1,\chi\chi')=r_p(E_1,\chi')+2\quad \text{and} \\
        &r_p(E_2,\chi\chi')=r_p(E_2,\chi').
    \end{align*}
    Then the result follows by induction.

    Let $N$ be the Galois closure of $L_{E_1,\chi'}K(\sqrt[p]{\Delta_2})$ over $K$, where $\Delta_2$ is a generator of $\A_2$. Observe that $N$ and $M_2$ are linearly disjoint.
    Choose a prime $\q\notin\Sigma(\chi')$ such that $\Frob_{\q}|_{N}=1$ and $(\Frob_\q|_{M_2})^p\ne 1$. Then $\q\in\P_2(E_1)\cap\P_0(E_2)$ and $\loc_{\q}(\Sel_p(E_1/K,\chi'))=0$. Then, by \cite[Proposition 7.2]{kmr2}, there exist $(p-1)$-characters $\psi_{\q}\in\cC(K_{\q})$ such that $r_p(E_1,\chi',\psi_{\q})=r_p(E_1,\chi')+2$, where $\Sel_p(E_1,\chi',\psi_{\q})$ is defined by the following exact sequence: $$0 \rightarrow \Sel_p(E_1,\chi',\psi_{\q}) \rightarrow H^1(K,E_1[p]) \rightarrow {\frac{H^1(K_{\q},E_1[p])} {\gamma_{\q}(\chi'_{\q}\psi_{\q})}}\dirsum{v\ne\q}{\displaystyle\frac{H^1(K_v,E_1[p])} {\gamma_v(\chi'{_v})}}.$$
    Let $\Sigma(\chi',\q):=\Sigma(\chi')\cup\{\q\}$.
    Define \begin{itemize}
        \item $Q:=(\Sigma(\chi',\q)\cup\P_0)^c,$
        \item $J:=\O_{K,\Sigma(\chi',\q)}^\times,$
        \item $G:=\prod_{\p\in\P_0}\O_\p^\times,$
        \item $ H:=\prod_{\p\in Q} \O_\p^\times \times \prod_{v\in\Sigma(\chi',\q)}K_v^\times.$
    \end{itemize} By Lemma \ref{elem}, the image of map $$\cC(K) \too \prod_{\p\in Q}\Hom(\O_\p^\times, \mu_p) \times \prod_{v\in \Sigma(\chi',\q)} \Hom(K_v^\times, \mu_p)$$ is equal to $$\{f\in\prod_{\p \in Q} \Hom(\O_\p^\times, \mu_p) \times \prod_{v \in \Sigma(\chi',\q)} \Hom(K_v^\times, \mu_p) : f(\A_3)=1 \},$$ where $$\A_3:=\ker(\O_{K,\Sigma(\chi',\q)}^\times/{(\O_{K,\Sigma(\chi',\q)}^\times)^p} \rightarrow \prod_{\q\in\P_0} \O_\q^\times/{(\O_\q^\times)^p}).$$
    Let $$\psi=(\psi_v) \in \prod_{\p \in Q} \Hom(\O_\p^\times, \mu_p) \times \prod_{v \in \Sigma(\chi',\q)} \Hom(K_v^\times, \mu_p)$$ be such that
    \begin{itemize}
        \item $\psi_\p$ is trivial for $\p\in Q$,
        \item $\psi_v=1_v$ for $v\in\Sigma(\chi')$,
        \item $\psi_{\q}$ is the character chosen above.
    \end{itemize}
    Then $\psi(\Delta_i)=1$ for $i=1,2$. By Lemma \ref{Alem}, $\psi(\A_1\A_2)=1$ implies that $ \psi(\A_3)=1.$ Thus, there exists ${\chi} \in \cC(K)$ such that
    \begin{itemize}
        \item ${\chi}_\p$ is unramified for $\p\in Q$,
        \item ${\chi}_v=1_v$ for $v\in\Sigma(\chi')$,
        \item ${\chi}_{\q}=\psi_{\q}$.
    \end{itemize} 
    Then $\Sel_p(E_1/K,\chi\chi')=\Sel_p(E_1/K,\chi',\psi_\q)$. Thus, $$r_p(E_1,\chi\chi')=r_p(E_1,\chi')+2.$$ Since $\q\in\P_0(E_2)$, $r_p(E_2,\chi\chi')=r_p(E_2,\chi')$.
\end{proof}

\begin{prop} \label{7prop2}
    Assume $1\ne[M_1:M_1\cap M_2]< [M_2:M_1\cap M_2]=p$.
    Then, for any positive integer $d$, there exist infinitely many $\chi\in\cC(K)$ such that $$r_p(E_2,\chi)=r_p(E_2)+2d \quad\text{and} \quad r_p(E_1,\chi)=r_p(E_1).$$
\end{prop}

\begin{proof}
    Fix $\chi'$. The proof is similar to the proof of Proposition \ref{7prop1}.
    Let $N$ be the Galois closure of $L_{E_2,\chi'}K(\sqrt[p]{\Delta_1})$ over $K$, where $\Delta_1$ is a generator of $\A_1$. Observe that $N$ and $M_1$ are linearly disjoint.
    Choose a prime $\q\notin\Sigma(\chi')$ such that $\Frob_{\q}|_{N}=1$ and $(\Frob_\q|_{M_1})^p\ne 1$. Then $\q\in\P_0(E_1)\cap\P_2(E_2)$ and $\loc_{\q}(\Sel_p(E_2/K,\chi'))=0$. By \cite[Proposition 7.2]{kmr2}, choose a local character $\psi_\q$ so that $r_p(E_2,\chi',\psi_\q)=r_p(E_2,\chi')+2$. Since $\q\in\P_0(E_1)$, $r_p(E_1,\chi',\psi_\q)=r_p(E_1,\chi')$. Since $\psi_\q(\Delta_i)=1$ for $i=1,2$, as in the proof of Proposition \ref{7prop1}, let $\chi$ be the global character such that \begin{itemize}
        \item ${\chi}_\p$ is unramified for $\p\in Q$,
        \item ${\chi}_v=1_v$ for $v\in\Sigma(\chi')$,
        \item ${\chi}_{\q}=\psi_{\q}$.
    \end{itemize} Then $r_p(E_2,\chi'\chi)=r_p(E_2,\chi')+2$ and $r_p(E_1,\chi'\chi)= r_p(E_1,\chi').$
\end{proof}

\begin{prop} \label{lastprop}
    Suppose that $M_1 \ne M_2$ and that $[M_i:K(\mu_p)]\nmid p$ for $i=1,2.$ Assume that $[M_2 : M_1\cap M_2]=p$ and that $[M_1:M_1\cap M_2]\mid p.$  Fix $\chi_0.$ Then there exist infinitely many characters $\chi\in\cC(K)$ such that $$r_p(E_1,\chi\chi_0)=r_p(E_1,\chi_0)+4 \quad \text{and}\quad r_p(E_2,\chi\chi_0)\le r_p(E_2,\chi_0)+2.$$ 
\end{prop}

\begin{proof}
    Fix $\chi_0$. Let $N$ be the Galois closure of $L_{E_1,\chi_0}L_{E_2,\chi_0}K(\sqrt[p]{\Delta_1},\sqrt[p]{\Delta_2})$ over $K$. Choose a prime $\q\notin\Sigma(\chi_0)$ such that $\Frob_\q|_N=1.$ Then $\q\in\P_2(E_1)\cap\P_2(E_2)$ and $\loc_\q(\Sel_p(E_i/K,\chi_0))=0$ for $i=1,2.$ Then, by \cite[Proposition 7.2]{kmr2}, choose a local character $\psi_\q$ so that $r_p(E_1,\chi_0,\psi_\q)=r_p(E_1,\chi_0)+2$. Since $\psi_\q(\Delta_i)=1$ for $i=1,2$, as in the proof of Proposition \ref{7prop1}, let $\chi_1$ be the global character such that \begin{itemize}
        \item ${\chi_1}_\p$ is unramified for $\p\in Q$,
        \item ${\chi_1}_v=1_v$ for $v\in\Sigma(\chi_0)$,
        \item ${\chi_1}_{\q}=\psi_{\q}$.
    \end{itemize} Then $r_p(E_1,\chi_1\chi_0)=r_p(E_1,\chi_0)+2.$ By Theorem \ref{mod2} and Theorem \ref{dual}, either $$r_p(E_2,\chi_1\chi_0)=r_p(E_2,\chi_0) \quad\text{or}\quad r_p(E_2,\chi_1\chi_0)=r_p(E_2,\chi_0)+2.$$
    Suppose that $r_p(E_2,\chi_1\chi_0)=r_p(E_2,\chi_0).$ Repeating this process, we have that there exists $\chi\in\cC(K)$ such that 
    \begin{align*}
        &r_p(E_1,\chi\chi_0)=r_p(E_1,\chi_0)+4 \quad \text{and}\\
        &r_p(E_2,\chi\chi_0)\le r_p(E_2,\chi_0)+2.    
    \end{align*}
    Now, suppose that $r_p(E_2,\chi_1\chi_0)=r_p(E_2,\chi_0)+2.$ Then, by Lemma \ref{4.6}, we have the following exact sequence:
    \begin{align*}
        0 \to \Sel_p(E_2/K,\chi_0) \to \Sel_p(E_2/K,\chi_1\chi_0) \to \Hom(G(L_\q/K_\q),E_2[p])\to 0.
    \end{align*}
    Let $\{P_1,P_2\}$ be a basis of $E_2[p]$ such that $E_2(M_1\cap M_2)=\langle P_1\rangle.$
    Let $t\in\Sel_p(E_2/K,\chi_1\chi_0)$ be such that $\image(\loc_\q(t))$ contains $P_2.$ Let $\tilde{t}$ be the image of $t$ in $H^1(M_1M_2,E_2[p])=\Hom(G_{M_1M_2},E_2[p]).$ Since $G(M_1M_2/M_1)\cong G(M_2/M_1\cap M_2)$ acts like $\begin{bmatrix}
        1 & 1 \\
        0 & 1
    \end{bmatrix}$ on $E_2[p]$ (with respect to basis $\{P_1, P_2\}$), and since $\tilde{t}$ is $G_K$-equivariant, $\tilde{t}$ is surjective.
    Let $F$ be the fixed field of $\ker(\tilde{t}).$ Then we have the following exact sequence $$0 \too E_2[p]\cong G(F/M_1M_2) \too G(F/M_1) \too G(M_1M_2/M_1) \too 0.$$ Note that $G(M_1M_2/M_1)\cong G(M_2/M_1\cap M_2).$
    Using a similar argument as in the proof of Proposition \ref{4.1}, one can show that $G(F/M_1)$ is not abelian. Thus, $F\not\subset M_1(\sqrt[p]{M_1^\times}).$ 
    Now, let $\chi':=\chi_1\chi_0$ and let $N$ be the Galois closure of $L_{E_1,\chi'} K( \sqrt[p]{\Delta_2}).$ Then $NM_2\subset M_1(\sqrt[p]{M_1^\times})$ and $F\not\subset M_1(\sqrt[p]{M_1^\times}).$ Choose $\q_1\notin\Sigma(\chi')$ such that $\Frob_{\q_1}|_{NM_2}=1$ and $\Frob_{\q_1}|_F\ne 1.$ Then $\loc_{\q_1}(\Sel_p(E_1/K,\chi'))=0$ and $\loc_{\q_1}(\Sel_p(E_2/K,\chi'))\ne 0.$ Choose a local character $\psi_{\q_1}$ such that $r_p(E_1,\chi',\psi_{\q_1})=r_p(E_1,\chi')+2$ and $r_p(E_2,\chi',\psi_{\q_1})\le r_p(E_2,\chi').$ Since $\psi_{\q_1}(\Delta_i)=1$ for $i=1,2$, let $\chi$ be the global character such that \begin{itemize}
        \item ${\chi}_\p$ is unramified for $\p\in Q$,
        \item ${\chi}_v=1_v$ for $v\in\Sigma(\chi')$,
        \item ${\chi}_{\q_1}=\psi_{\q_1}$.
    \end{itemize} 
    Then $r_p(E_1,\chi\chi')=r_p(E_1,\chi')+2$ and $r_p(E_2,\chi\chi')\le r_p(E_2,\chi').$
    Therefore, there exist infinitely many characters $\chi\in \cC(K)$ such that $$r_p(E_1,\chi\chi_0)=r_p(E_1,\chi_0)+4 \quad \text{and}\quad r_p(E_2,\chi\chi_0)\le r_p(E_2,\chi_0)+2.$$
\end{proof}

\begin{thm} \label{thm4}
    Suppose that $M_1\ne M_2$, and that $[M_i:K(\mu_p)]\ne p$ for $i=1,2$. Then, for any positive integer $d$, there are infinitely many $\chi\in\cC(K)$ such that $$|r_p(E_1,\chi)-r_p(E_2,\chi)|\ge |r_p(E_1)-r_p(E_2)|+2d.$$
\end{thm}

\begin{proof}
    Assume that $[M_1:K(\mu_p)] \le [M_2:K(\mu_p)].$ If $[M_2:M_1\cap M_2]\ne p$, then apply Proposition \ref{7prop1} and use induction. Suppose that $[M_2:M_1\cap M_2]=p$. If $[M_1:M_1\cap M_2]\nmid p$, apply Proposition \ref{7prop2} and use induction. If $[M_1:M_1\cap M_2]\mid p$, apply Proposition \ref{lastprop} and use induction. 
\end{proof}

\section*{Acknowledgments}
The author is grateful to his PhD advisor, Myungjun Yu, for the guidance and assistance provided during the course of this research. The author also thanks Hyungmin Jang for many helpful discussions and valuable comments. Minseok Kim was supported by the National Research Foundation of Korea (NRF) grant funded by the Korea government (MSIT) (RS-2025-23525445).


\begin{thebibliography}{10}

\bibitem{sscc}
C.-H. Chiu.
\newblock Strong {S}elmer companion elliptic curves.
\newblock {\em J. Number Theory}, 217:376--421, 2020.

\bibitem{msk}
M.~Kim.
\newblock Increasing the {$p$}-{S}elmer rank by twisting.
\newblock {\em Int. J. Number Theory}, 21(10):2573--2585, 2025.

\bibitem{kmr1}
Z.~Klagsbrun, B.~Mazur, and K.~Rubin.
\newblock Disparity in {S}elmer ranks of quadratic twists of elliptic curves.
\newblock {\em Ann. of Math. (2)}, 178(1):287--320, 2013.

\bibitem{kmr2}
Z.~Klagsbrun, B.~Mazur, and K.~Rubin.
\newblock A {M}arkov model for {S}elmer ranks in families of twists.
\newblock {\em Compos. Math.}, 150(7):1077--1106, 2014.

\bibitem{kolysys}
B.~Mazur and K.~Rubin.
\newblock Kolyvagin systems.
\newblock {\em Mem. Amer. Math. Soc.}, 168(799):viii+96, 2004.

\bibitem{alc}
B.~Mazur and K.~Rubin.
\newblock Finding large {S}elmer rank via an arithmetic theory of local constants.
\newblock {\em Ann. of Math. (2)}, 166(2):579--612, 2007.

\bibitem{scc}
B.~Mazur and K.~Rubin.
\newblock Selmer companion curves.
\newblock {\em Trans. Amer. Math. Soc.}, 367(1):401--421, 2015.

\bibitem{ds}
B.~Mazur and K.~Rubin.
\newblock Diophantine stability.
\newblock {\em Amer. J. Math.}, 140(3):571--616, 2018.
\newblock With an appendix by Michael Larsen.

\bibitem{MRS}
B.~Mazur, K.~Rubin, and A.~Silverberg.
\newblock Twisting commutative algebraic groups.
\newblock {\em J. Algebra}, 314(1):419--438, 2007.

\bibitem{serre}
J.-P. Serre.
\newblock Propri\'et\'es galoisiennes des points d'ordre fini des courbes elliptiques.
\newblock {\em Invent. Math.}, 15(4):259--331, 1972.

\bibitem{2snc}
M.~Yu.
\newblock 2-{S}elmer near-companion curves.
\newblock {\em Trans. Amer. Math. Soc.}, 372(1):425--440, 2019.

\end{thebibliography}
\end{document}